\documentclass[12pt,a4paper,oneside]{amsart}
\usepackage{hyperref}
\usepackage{paralist, mathdots} %

\usepackage{amssymb,tabularx,nicematrix} %
\usepackage{amsmath,tikz,multirow,lscape,subcaption,ytableau,pifont}

\usepackage{centernot}

\usepackage{amsthm}
\usepackage{soul, arydshln}
\usepackage{scalerel}

\newtheorem{proposition}{Proposition}[section]

\newtheorem{theorem}[proposition]{Theorem}

\newtheorem{corollary}[proposition]{Corollary}
\newtheorem{lemma}[proposition]{Lemma}
\theoremstyle{definition}
\newtheorem{example}[proposition]{Example}
\newtheorem{remark}[proposition]{Remark}
\newtheorem{definition}[proposition]{Definition}
\newtheorem{question}[proposition]{Question}

\newcommand{\bR}{\mathbb{R}}
\newcommand{\bN}{\mathbb{N}}

\newcommand{\A}{\mathcal{A}}
\newcommand{\bdelta}{\boldsymbol{\delta}}
\newcommand{\bepsilon}{\boldsymbol{\epsilon}}

\newcommand{\mc}{\mathcal}
\DeclareMathOperator{\iso}{\iota}

\newcommand{\1}{{\bf 1}}

\newcommand{\bfb}{{\bf b}}
\newcommand{\bfc}{{\bf c}}
\newcommand{\bfd}{{\bf d}}
\newcommand{\bfr}{{\bf r}}
\newcommand{\bfp}{{\bf p}}
\newcommand{\bfq}{{\bf q}}

\let\epsilon\varepsilon

\newcommand{\tuple}{\mathbf}
\newcommand{\m}{\tuple{m}}
\newcommand{\n}{\tuple{n}}

\newcommand{\trans}{^\top}

\newcommand{\LL}{\mathcal{D}}

\tikzset{every picture/.append style={scale=0.6}}

    \definecolor{helena}{rgb}{.2,.8,.4}
    \definecolor{polona}{rgb}{.2,.2,.8}
    \definecolor{rupert}{rgb}{0,.5,.5}
   \definecolor{todo}{rgb}{.8,.2,.2}

\AtBeginDocument{%
   \def\MR#1{}
}

\begin{document}
\title{Orthogonalisability of joins of graphs}
\author{Rupert H. Levene, Polona Oblak, Helena \v Smigoc}
\address[R.~H.~Levene and H.~\v Smigoc]{School of Mathematics and Statistics, University College Dublin, Ireland}
\email{rupert.levene@ucd.ie}\email{helena.smigoc@ucd.ie}
\address[P.~Oblak]{University of Ljubljana, Faculty of Computer and Information Science
and Faculty of Mathematics and Physics, Slovenia; Institute of Mathematics, Physics, and Mechanics, Slovenia}
\email{polona.oblak@fri.uni-lj.si}

\bigskip

\begin{abstract}
A graph is said to be orthogonalisable if the set of real symmetric matrices whose off-diagonal pattern is prescribed by its edges contains an orthogonal matrix. We determine some necessary and some sufficient conditions on the sizes of the connected components of two graphs for their join to be orthogonalisable.  In some cases, those conditions coincide, and we present several families of joins of graphs that are orthogonalisable. %
\end{abstract}

\maketitle

\section{Introduction}
Let $G=(V(G),E(G))$ be a simple graph with vertex set $V(G)=\{1,\dots,n\}$, and consider $S(G)$, the set of all real symmetric $n \times n$ matrices $A=(a_{ij})$ such that, for $i \neq j$, $a_{ij} \neq 0$ if and only if $\{i,j\} \in E(G)$. There are no restrictions on the diagonal entries of $A$.
The minimum number of distinct eigenvalues of a graph
$$q(G) = \min\{q(A)\colon A \in S(G)\},$$
where $q(A)$ denotes the number of distinct eigenvalues of a square matrix~$A$, is one of the parameters motivated by the Inverse Eigenvalue Problem for Graphs. The origins of the study of this parameter can be found in several key texts \cite{MR347649,MR572262,MR1220704}, and many papers have been published on the subject~\cite{2013-ELA-Ahmadi-Cavers-Fallat-q,MR3665573,18-Bjorkman-q,MR2735867,2022-Hogben-Lin-Shader-IEPG-book,MR1899084,MR3891770}.

Since the diagonal elements of matrices in $G$ are not restricted, it follows that, provided $G$ has at least one edge, we have $q(G)=2$ if and only if $S(G)$ contains an orthogonal symmetric matrix. For this reason, we say in this case that $G$ is \emph{orthogonalisable}.
Although there is no general characterisation of orthogonalisable graphs, several families of connected orthogonalisable graphs are known
(see, e.g.,~\cite{2013-ELA-Ahmadi-Cavers-Fallat-q,2019-ELA-Bailey-Craigen,2023-LAA-Barrett-Fallat-regular-graphs-q=2,2024-Fallat-distance-regular,2022-LAA-Fallat-Mojallal,2018-LAA-Johnson-Zhang, MR3891770}), %
 and several necessary conditions %
have been determined~\cite{2017-Special-Matrices-Adm-Fallat-q=2,corrigendum,2013-ELA-Ahmadi-Cavers-Fallat-q,2023-LAA-Barrett-Fallatt-sparsity-q=2}. It is suspected that the complete characterisation of orthogonalisable graphs is a difficult problem.

The minimum number of distinct eigenvalues of joins of graphs has been investigated in several works, e.g., \cite{2023-Abiad-Bordering-of-symmetric-matrices,2017-Special-Matrices-Adm-Fallat-q=2,2013-ELA-Ahmadi-Cavers-Fallat-q,levene2020orthogonal}. Recall that the join of two graphs $G$ and $H$, denoted $G\vee H$, is the graph formed by adding all possible edges between $G$ and $H$ to the disjoint union $G\cup H$. It is known that $q(G \vee H)=2$ for all connected graphs $G$ and $H$ such that $\big||G|-|H|\big| \leq 2$,~\cite[Theorem~3.4]{levene2023distinct}; see also Ahmadi~et~al.~\cite[Theorem~5.2]{2013-ELA-Ahmadi-Cavers-Fallat-q} for the case $G=H$ and Monfared and Shader~\cite[Theorem~5.2]{MonfaredShader2016-Nowhere-zero-eigenbasis} for  $|G|=|H|$. This implies that there does not exist any set of forbidden subgraphs of graphs with $q=2$.  It was shown in~\cite[Example~3.5]{levene2023distinct} that for trees $G$ and $H$ we have $q(G \vee H) = 2$ if and only if $\big||G|-|H|\big|\leq 2$. In particular, $q(P_m\vee P_n)\geq 3$ for $|m-n|\geq 3$, where $P_n$ is the path on $n$ vertices.
Recently, joins of disconnected graphs were investigated in~\cite{2017-Special-Matrices-Adm-Fallat-q=2,corrigendum,levene2020orthogonal,levene2023distinct}, where particular attention was given to joins of unions of complete graphs and paths.

\subsection{Notation and terminology}
We write $\bN=\{1,2,3,\dots\}$ and $\bN_0=\{0\}\cup \bN$. For $k\in \bN$, we write $[k]=\{1,2,\dots,k\}$. Vectors are written in boldface, and the $i$th component of a vector $\mathbf{v}$ is written as $v_i$. 
For integer tuples $\mathbf{x}\in \bN_0^t$ we write $\iota(\mathbf{x})=|\{i\in [t]:x_i=1\}|$ and $|\mathbf{x}|=\sum_{i\in [t]}x_i$. If $x_i\ge x_{i+1}$ for all $i<t$, then we say that $\mathbf{x}$ is non-increasing. When necessary, we extend integer tuples by appending zero parts at the end, and adopt the useful shorthand %
$(k^s):=(k,k,\dots,k)$, where $k$ appears $s$ times. At times we will use this notation for subtuples, too, so that, for example $(4,3^2,1^5)=(4,3,3,1,1,1,1,1)$. We use standard notation for conjugate partitions: for $\mathbf{x}\in \bN_0^\ell$, we set
$x_j^*=|\{ i\in [\ell] : x_i\ge j\}|$ and $\mathbf{x}^*=(x_j^*)$. 
  We write $\preceq_w$ for the weak majorisation order on non-increasing integer tuples, defined by %
  $\mathbf{x} \preceq_w \mathbf{y}$ if and only if $\sum_{i\in [k]} x_i\leq \sum_{i\in [k]} y_i$ for all $k\in \bN$. If in addition $|\mathbf{x}|=|\mathbf{y}|$, then we say that $\mathbf{y}$ majorises $\mathbf{x}$ and write $\mathbf{x} \preceq \mathbf{y}$.

The notation $\mathbf{e}_j$ is used for the $j$th standard basis vector (whose length will be clear from the context) with every entry equal to $0$ except for its $j$th entry, which is $1$, and $\mathbf{1}_s$ is the vector with $s$ entries, all equal to~$1$. If $r,s\in \bN$ and $\mc S\subseteq\bR$, then $\mc S^{r\times s}$ is the set of $r\times s$ matrices with entries in~$\mc S$. The transpose of a matrix $X$ is written as $X\trans$. If $X=(x_{ij})$ and $Y=(y_{ij})$ are real matrices of equal size and $x_{ij}\le y_{ij}$ for every $i,j$, then we write $X\le Y$. The notation $\mathbf{v}\le \mathbf{w}$ for real vectors $\mathbf{v}$ and $\mathbf{w}$ of equal size has the same meaning.

The graphs considered in this work are all finite, simple, undirected graphs with at least one vertex. We say that a graph $G$ is \emph{sized by} $\m\in \bN^k$ if $G$ has $k$ connected components with sizes $m_1,\dots,m_k$, respectively, and a pair of graphs $(G,H)$ is \emph{sized by} a pair $(\m,\n)\in \bN^k\times \bN^\ell$ if $G$ is sized by~$\m$ and $H$ is sized by~$\n$.  Note that $\iota(\n)$ is the number of isolated vertices of $H$. %

\subsection{Overview of the main results} 
In this work, we investigate pairs $(\m,\n) \in \bN^k \times \bN^\ell$ for which any pair of graphs $(G, H)$ sized by $(\m,\n)$ satisfies $q(G \vee H) = 2$. Our main results are summarised in the following theorem. 

\begin{theorem}\label{thm:main:nec-suff-q=2}
Let $k\leq \ell$, $\psi=\lceil \ell/k\rceil$ and $(\m,\n)\in \bN^k\times \bN^\ell$.
The following statements satisfy the implications $\text{(i)}\implies\text{(ii)}\iff\text{(iii)}\implies\text{(iv)}$:

\begin{enumerate}[(i)]
    \item $|\n|-n_2^*-n_3^*\le |\m|\le |\n|+m_{\psi+1}^*+m_{\psi+2}^*$, $\m\geq (\psi^k)$ and $$\sum_{j=1}^\psi n^*_j\ge  k\psi+\max\{0,|\n|-|\m|-n^*_{\psi+1}-n^*_{\psi+2}\}.$$
    \item  $q(G\vee H)=2$ for all graphs $G,H$ which are sized by $(\m,\n)$. 
    \item $q(G\vee H)=2$ for $G=\bigcup_{i=1}^k P_{m_i}$ and $H=\bigcup_{j=1}^\ell P_{n_j}$.
    \item $|\n|-n_2^*-n_3^*\le |\m|\le |\n|+m_{\psi+1}^*+m_{\psi+2}^*$, $\m\geq (\psi^k)$, and if $k<\ell$, then we also have \[  \iota(\n)\le 
  \max\left\{(\psi-1)k,\frac{\psi}{\psi-1}(\ell-k)\right\}.\]
\end{enumerate}
\end{theorem}

\begin{figure}[htb] 
 \tikzset{
    every node/.style={draw, circle, fill=white, inner sep=1pt}
    }
\begin{subfigure}{0.45\textwidth}  
  \centering
\begin{tikzpicture}[scale=2]
\foreach \x in {-1.5,-0.5,0.5,1.5}{
  \foreach \y in {-1.5,-0.5,0.5,1.5}{
  \draw[color=gray] (0,\x)--(2,\y); 
    }
    }
 \draw[thick] (2,0.5)--(2,1.5);
\foreach \x in {-1.5,0.5}{
            \draw[thick] (0,\x)--(0,\x+1);
             \node at (0,\x) { };
             \node at (0,\x+1) { };
             \node at (2,\x) { };
             \node at (2,\x+1) { };
    }
\end{tikzpicture}
    \caption{$2P_2 \vee (P_2\cup 2P_1)$}
     \label{fig:not-AJO1}
    \end{subfigure}
\begin{subfigure}{0.45\textwidth}  
  \centering
\begin{tikzpicture}[scale=2]
\foreach \x in {-0.5,0.5,1.5}{
  \foreach \y in {-1.5,-0.5,0.5,1.5}{
  \draw[color=gray] (0,\x)--(2,\y); 
    }
    }
\draw[thick] (2,0.5)--(2,1.5);
\draw[thick] (0,0.5)--(0,1.5);
\foreach \x in {-0.5,0.5,1.5}{
             \node at (0,\x) { };
             \node at (2,\x) { };
    }
\node at (2,-1.5) { };
\end{tikzpicture}
\caption{$(P_2 \cup P_1) \vee (P_2\cup 2P_1)$}  \label{fig:SJO2}
\end{subfigure}
\caption{The graph in (a) is sized by $(\m,\n)=((2^2),(2,1^2))$ which satisfies condition (i) in Theorem~\ref{thm:main:nec-suff-q=2}, hence it has $q=2$. The graph in (b) is sized by $(\m,\n)=((2,1),(2,1^2))$ which fails condition (iv) in Theorem~\ref{thm:main:nec-suff-q=2}, so this graph has $q>2$.}\label{fig:intro-q-2}
\end{figure}

This %
theorem %
will follow from Proposition~\ref{prop:weakly-suitable-without-delta-epsilon}, Theorem~\ref{thm:main}, Remark~\ref{remark:SJO-K-AJO-P} and Corollary~\ref{cor:paths-q>2}.

We will prove in Corollary~\ref{cor:k-divides-l-or-l+1} that when $k$ divides $\ell$ or $\ell+1$, all four statements in Theorem \ref{thm:main:nec-suff-q=2} are equivalent, but in general, (ii) does not imply (i) and (iv) does not imply (ii), see Proposition~\ref{prop:strict-inclusions}.
Example~\ref{example:k=3,l=4} shows that (ii) and (iv) are also equivalent  for $k=3$ and $\ell=4$.

\section{Compatible multiplicity matrices}

\label{subsec:compatible-multiplicity-matrices}
Let $G$ be a graph with $k$ connected components $G_1,\dots,G_k$. A matrix $V$ with non-negative integer entries \emph{fits} $G$ if $V$ has $k$ columns and the $i$th column of $V$ has sum $|G_i|$ for $1\le i\le k$.
We say that $V$ is a \emph{multiplicity matrix} for $G$ if there is a matrix $A\in S(G)$ and real numbers $\lambda_1<\dots<\lambda_r$, where $r$ is the number of rows of $V$, such that each matrix entry $v_{i,j}$ is the multiplicity of $\lambda_i$ as an eigenvalue of $A_j$, the submatrix of $A$ corresponding to the connected component $G_j$ of $G$. In particular, this implies that the $j$th column of $G$ is an ordered multiplicity list of some matrix in $S(G_j)$.
(The zero entries in the matrix correspond to multiplicity zero, and are ignored in the corresponding ordered multiplicity list of $A_j$.)
By a $0$-$1$ matrix, we mean a matrix whose entries are all in $\{0,1\}$. Monfared and Shader,~\cite[Corollary~4.3]{MonfaredShader2013-interlacing} and~\cite[Theorem~4.3]{MonfaredShader2016-Nowhere-zero-eigenbasis}, showed  that any $0$-$1$ matrix which fits a graph $G$ is a multiplicity matrix for $G$.

In \cite{levene2020orthogonal} the authors defined the notion of compatible multiplicity matrices and proved that the existence of compatible multiplicity matrices that fit graphs $G$ and $H$ is necessary for $q(G\vee H)=2$. Given a matrix $X$ with at least $3$ rows, we denote by $\widetilde X$ the matrix obtained by deleting the first and last rows of~$X$. 
We say that matrices $V \in \bN_0^{r \times k}$ and $W \in \bN_0^{r \times \ell}$ are \emph{compatible} if $r\ge3$ and
  $$\widetilde V\1_k=\widetilde W\1_\ell \; \text{ and } \; \widetilde V\trans \widetilde W \in \bN^{k \times \ell}.$$
  In other words, the row-sums of $\widetilde{V}$ and $\widetilde{W}$ agree, and no column of $\widetilde{V}$ is orthogonal to a column of $\widetilde{W}$.

\begin{theorem}[{\cite[Corollaries~3.2 and~3.3]{levene2023distinct}}]\label{thm:0-1-compatible->q=2}
 If there exist compatible $0$-$1$ matrices that fit graphs $G$ and $H$, then $q(G\vee H)=2$.
 Moreover, if $G$ and $H$ are unions of paths, then the converse holds.
\end{theorem}

  For $r,s\in \bN$ and $\bfr\in \bN_0^r$, $\mathbf{s}\in \bN_0^s$,
  we use the standard notation $\A(\bfr,\mathbf{s})$ for the set of $r\times s$ matrices $X\in \{0,1\}^{r\times s}$ such that $\mathbf{r}$ and $\mathbf{s}$ are the row-sum and the column-sum of $X$, respectively. That is, $X\mathbf{1}_{s}=\mathbf{r}$ and $X\trans\mathbf{1}_r=\mathbf{s}$.
\begin{definition}
  Let $k,\ell\in \bN$. We define sets
\begin{align*}
    \LL(k,\ell):=\{(\m,\n)\in {\bN^k} \times {\bN^{\ell}} & \colon   \text{there exist  $s\in \bN$, $\bfr_V,\bfr_W\in \bN_0^{s+2}$ and}   \\
     &\text{\;\;\;compatible $V\in {\mathcal A}(\bfr_V,\m)$ and  $W\in{\mathcal A}(\bfr_W,\n)$}
       \}
\end{align*}
and
\begin{align*}
\widetilde{\LL}(k,\ell):=\{(\m,\n)\in {\bN^k} \times {\bN^{\ell}} \colon  &\text{there exist } s\in\bN,\; \bfr\in \bN_0^s \text{ and}\\
  &V\in {\mathcal A}(\bfr,\m), \;
  W\in{\mathcal A}(\bfr,\n)
  \text{ with } V\trans W\in \bN^{k \times \ell}\}.
\end{align*}
\end{definition}

\begin{remark}\label{remark:D-sets-and-AJO}\begin{enumerate}[(a)]
\item By symmetry, we may restrict attention to these sets in the case $k\le \ell$.
\item \label{remark:D-vs-tilde(D)} We have $(\m,\n)\in \LL(k,\ell)$ if and only if there exist $\bdelta\in \{0,1,2\}^k$ and $\bepsilon\in \{0,1,2\}^{\ell}$ such that $(\m-\bdelta,\n-\bepsilon)\in \widetilde{\LL}(k,\ell)$. Indeed, if $(\m-\bdelta,\n-\bepsilon)\in \widetilde{\LL}(k,\ell)$, with corresponding matrices $V,W$, then we can write $\bdelta$ as the sum of two $0$-$1$ vectors and extend $V$ by appending these as a new first and final row, and similarly extend $W$ using $\bepsilon$. These extended matrices show that $(\m,\n)\in \LL(k,\ell)$. Conversely, given compatible $V,W$ showing that $(\m,\n)\in \LL(k,\ell)$, let $\bdelta$ and  $\bepsilon$ be the sum of the first and final rows of $V$ and $W$, respectively. By compatibility, we have $\bfr:=\widetilde{\bfr_V}=\widetilde{\bfr_W}$, and since $\m':=\m-\bdelta$ and $\n':=\n-\bepsilon$ are the column sums of $\widetilde V$ and $\widetilde W$, respectively, by compatibility $\m'$ and $\n'$ have no zero entries, so $(\m',\n')\in \bN^k\times \bN^\ell$.
Moreover,  $\widetilde V\in \A(\bfr,\m')$ and $\widetilde W\in \A(\bfr,\n')$, so $(\m',\n')\in \widetilde{\LL}(k,\ell)$.
\item \label{remark:D-implies-same-size} If $(\m,\n)\in \widetilde{\LL}(k,\ell)$, then $|\m|=|\n|$. Indeed, if $V\in \mathcal{A}(\mathbf{r},\m)$ and $W\in \mathcal{A}(\mathbf{r},\n)$, then the sums of all entries of $V$ are given by $|\mathbf{r}|=|\m|$, and similarly for $W\in \mathcal{A}(\mathbf{r},\n)$, so this is immediate from the definition above. The same conclusion does not necessarily hold for $(\m,\n)\in \LL(k,\ell)$; for example, if $(\m,\n)=((3,2),(2,2))$, then $(\m,\n)\in \LL(2,2)$ and $|\m|\ne |\n|$.
\end{enumerate}
\end{remark}

The following results give some sufficient conditions and some necessary conditions for membership in $\widetilde{\LL}(k,\ell)$ that are motivated by  Theorem \ref{thm:0-1-compatible->q=2}, and are derived by considering the existence of compatible $0$-$1$ matrices. 

\begin{lemma}\label{lem:new1}
   Let $1\le k\le \ell$ and $s\ge1$.
  If $(\m,\n)\in \widetilde{\LL}(k,\ell)$,
  then
  for any $\mathbf{p}\in \bN_0^k$, $\mathbf{q}\in \bN_0^\ell$ with $|\bfp|=|\bfq|$, we have $(\m+\mathbf{p},\n+\mathbf{q})\in \widetilde{\LL}(k,\ell)$.
\end{lemma}
\begin{proof}
  Let $(\m,\n)\in \widetilde{\LL}(k,\ell)$, so that there exist $s\in\bN$, $\bfr\in \bN_0^s$, $V\in {\mathcal A}(\bfr,\m)$ and $W\in{\mathcal A}(\bfr,\n)$ with $V\trans W\in \bN^{k \times \ell}$. If $|\bfp|=|\bfq|=1$, then
  let $X$ be $V$ with the extra row $\bfp$ appended, and let $Y$ be $W$ with the extra row $\bfq$ appended. Then $X\in \A((\bfr,1),\m+\bfp)$, $Y\in \A((\bfr,1),\n+\bfq)$ and $X\trans Y\ge V\trans W$
  so $X\trans Y\in \bN^{k\times \ell}$. Hence $(\m+\bfp,\n+\bfq)\in \widetilde{\LL}(k,\ell)$.
  If $|\bfp|=|\bfq|>1$, then we can decompose $\bfp$ and $\bfq$ as sums of standard basis vectors, pair them up in an arbitrary fashion and repeat the above argument $|\bfp|=|\bfq|$ times.
\end{proof}

\begin{lemma}\label{lem:new2}
  Let $1\le k\le \ell$ and $s\ge1$.
  If $\bfc\in \bN^\ell$ with $\bfc\le (s^\ell)$ and $|\bfc|=ks$,
  then %
  $((s^k),\bfc)\in \widetilde{\LL}(k,\ell)$.
\end{lemma}

\begin{proof}
 Since $\bfc\le (s^\ell)$ and $|\bfc|=ks=|(s^k)|$, the decreasing rearrangement of $\bfc$ is majorised by $(s^k)$, which has conjugate partition $(s^k)^*=(k^s)$.
  By the Gale-Ryser theorem~\cite{MR91855, MR0150048} (see also~\cite[Theorem~2.1.3]{BrualdiCMC2006}), there is an $s\times \ell$ matrix $W\in \A((k^s),\bfc)$. Let $V$ be the $s\times k$ matrix with every entry equal to $1$. Then $V\in \A((k^s),(s^k))$ and
  every entry of $V\trans W$ is an entry of $\bfc\in \bN^\ell$, so $((s^k),\bfc)\in \widetilde{\LL}(k,\ell)$.
\end{proof}

\begin{lemma}\label{lem:necessary-condition-on-m}
  If $k \leq \ell$ are positive integers and  $(\m,\n) \in \widetilde{\LL}(k,\ell)$, then $\m \geq (\psi^k)$
    where $\psi=\left\lceil\ell/k\right\rceil$. %
  \end{lemma}
  \begin{proof}
    Without loss of generality, suppose that the entries of $\m$ are arranged in decreasing order. We must show that $m_k\ge \psi$, or equivalently, $km_k\ge \ell$. There exist $\bfr$ and matrices $V\in{\mathcal A}(\bfr,\m)$, $W\in{\mathcal A}(\bfr,\n)$ such that $V\trans W$ is nowhere zero. By permuting the rows of $V$ and $W$ if necessary, we may assume that the $k$th column of $V$ is $(1^{m_k},0,0,\dots)\trans$.
    Since $V$ and $W$ are compatible, each of the $\ell$ columns of $W$ must contain at least one $1$ in the first $m_k$ rows; for otherwise, $V\trans W$ would have a zero entry. On the other hand, each row-sum $r_i$ of $W$ is equal to the corresponding row-sum of $V$, which  is at most $k$. Therefore
    \[\ell \leq \sum_{i\in [m_k]}r_i\leq km_k.\qedhere\]
\end{proof}

We now turn to a necessary
condition  for membership in $\widetilde{\LL}(k,\ell)$ which involves the number of isolated vertices (see
Proposition~\ref{prop:necc2}). For natural numbers $k<\ell$, let
\[
  \iota_{\max}(k,\ell):=
  \max\left\{(\psi-1)k,\frac{\psi}{\psi-1}(\ell-k)\right\}
\]
where $\psi=\lceil \ell/k\rceil$.

\begin{lemma}\label{lem:neccesary}
  Suppose $k< \ell$ and $s\in \bN$. Let $\bfc\in\bN_0^\ell$.
  If $\A((k^s),\bfc)\ne \emptyset$, then
  $\iota(\bfc)\le \iota_{\max}(k,\ell)$.
\end{lemma}
\begin{proof}
  Let $\psi=\lceil \ell/k\rceil$. Under the given hypotheses, we will show that if $\iota(\bfc)>(\psi-1)k$
  then we have
  $\iota(\bfc)\le \tfrac{\psi}{\psi-1}(\ell-k)$.
  Since $\A((k^s),\bfc)\ne \emptyset$, by
  the Gale-Ryser theorem, $\bfc$ is majorised by $(k^s)^*=(s^k)$,
  which implies that $|\bfc|=ks$ and no entry of $\bfc$
  exceeds~$s$. Hence
  \[ \iota(\bfc)\le|\bfc|=ks\le (\ell-\iota(\bfc))s+\iota(\bfc).\]
  Rearranging and using $\iota(\bfc)>(\psi-1)k\ge \ell-k$, which implies that $\iota(\bfc)+k-\ell>0$, we obtain
  \[ \frac{\iota(\bfc)}k\le s\le
    \frac{\iota(\bfc)}{\iota(\bfc)+k-\ell}.\] Since
  $(\psi-1)k<\iota(\bfc)$ and $s$ is an integer, the left-hand
  inequality yields $s\ge \psi$. The right-hand inequality now implies
  that $(\iota(\bfc)+k-\ell)\psi\le \iota(\bfc)$ and hence
  $\iota(\bfc)\le \tfrac{\psi}{\psi-1}(\ell-k)$.
\end{proof}

\begin{proposition}\label{prop:necc2}
  If $k <\ell$ and $(\m,\n)\in \widetilde{\LL}(k,\ell)$,
  then $\iota(\n)\le \iota_{\max}(k,\ell)$.
\end{proposition}
\begin{proof}
  We may assume that $\iota(\n)>0$ and $\n=(\n',1^{\iota(\n)})$ where
  $\n'\in \bN^{\ell-\iota(\n)}$. There exist $h\in \bN$,
  $\bfr\in \bN_0^h$ and matrices $V\in \A(\bfr,\m)$,
  $W\in \A(\bfr,\n)$ such that $V\trans W\in \bN^{k\times
    \ell}$. Write $W=(W'\;W'')$ where $W''$ has $\iota(\n)$
  columns. Then $W''$ has a single non-zero entry in each column,
  which together cover $s$ rows, say, where $s\in [h]$. By permuting
  the rows of $V$ and $W$, we may assume that the non-zero entries of
  $W''$ occur precisely in rows $1,\dots,s$. Then we can partition $W$
  as $W=\begin{pmatrix}
    W_1\\
    W_2
  \end{pmatrix}$, where $W_1$ is an $s\times \ell$ 0-1 matrix whose
  last $\iota(\n)$ columns include the vectors
  $\mathbf{e}_1,\dots,\mathbf{e}_s$. Since $V\trans W$ has no zero
  entries, it follows that
  $V=\begin{pmatrix}\mathbf{1}_s\mathbf{1}_k\trans\\V_2\end{pmatrix}$ for some
  matrix $V_2$, hence the row sums of $W_1$ are $(k^s)$. So
  $W_1\in \A((k^s),\bfc)$ for some $\bfc\in \bN_0^{\ell}$. The last $\iota(\n)$ column-sums of $W_1$ are all $1$, by construction, so
  $\iota(\bfc)\ge \iota(\n)$. By
  Lemma~\ref{lem:neccesary}, we have
  $\iota(\n)\le \iota(\bfc)\le \iota_{\max}(k,\ell)$.
\end{proof}

\section{Join-orthogonalisability, strong and weak suitability}\label{sec:suitability}

In this section, we establish one of our main results, Theorem~\ref{thm:main},
which relates two simple combinatorial conditions to two properties
related to the orthogonalisability of joins of graphs. We start by
defining and discussing these orthogonalisability properties.

\begin{definition}
  Let  $(\m,\n)\in \bN^k\times \bN^\ell$. We say that the pair $(\m,\n)$ is \emph{always join-orthogonalisable} if for all pairs of graphs $(G,H)$ that are sized by $(\m,\n)$, we have $q(G\vee H)=2$.

  We say that $(\m,\n)$ is \emph{sometimes join-orthogonalisable} if there is some pair of graphs $(G,H)$ that is sized by $(\m,\n)$, so that $q(G\vee H)=2$.
\end{definition}

\begin{remark}\label{remark:D-AJO} By Theorem~\ref{thm:0-1-compatible->q=2}, %
  the set of always join-orthogonalisable pairs in $\bN^k\times\bN^\ell$ coincides with $\LL(k,\ell)$. In principle, this gives a combinatorial characterisation of always join-orthogonalisability in terms of the existence or non-existence of certain pairs of $0$-$1$ compatible multiplicity matrices. However, determining this existence or non-existence seems a difficult combinatorial problem, and our motivation for the current work was to find arithmetic conditions on a given pair $(\m,\n)$ which make this simpler to check.
\end{remark}

\begin{remark}\label{remark:SJO-K-AJO-P}
  By \cite[Theorem~3.1 and Corollary~3.3]{levene2023distinct}, a pair $(\m,\n)\in \bN^k\times \bN^\ell$ is always join-orthogonalisable if and only if
$$q\left(\bigcup_{i\in [k]}P_{m_i}\vee \bigcup_{j\in [\ell]}P_{n_j}\right)= 2.$$
Recall \cite[Theorem~3.1]{2013-Barrett-complete-graphs} that any realisable eigenvalue list of a connected graph $G$ of order $n$ is also realisable for the complete graph $K_n$ of the same order. It follows from \cite[Theorem~3.1 and Corollary~3.2]{levene2023distinct} that $(\m,\n)\in \bN^k\times \bN^\ell$ is sometimes join-orthogonalisable if and only if
$$q\left(\bigcup_{i\in [k]}K_{m_i}\vee \bigcup_{j\in [\ell]}K_{n_j}\right)= 2.$$
\end{remark}

In light of this, we can state \cite[Theorem~4.10]{levene2020orthogonal} as the following characterisation of sometimes join-orthogonalisability.
\begin{theorem}[\cite{levene2020orthogonal}]\label{thm:join-union-complete-graphs-2}
  Let $k\le \ell$, $\m=(m_i)\in \bN^k$ and $\n=(n_j)\in\bN^\ell$. The pair $(\m,\n)$ is sometimes join-orthogonalisable if and only if one of the following is true:
\begin{enumerate}[(i)]
\item $\iso(\m)=0$ and $\iso(\n)=0$ and $\ell\le|\m|$;
\item $\iso(\m)\ne 0$ and $k+\ell\le |\m|+\iso(\m)$;
\item $\iso(\m)=0$ and $\iso(\n)\ne 0$ and either $k+\ell\le|\m|$, or $2k\le \ell \le |\m|$, or $\ell\le 2k\le |\n|$.
\end{enumerate}
\end{theorem}

We now define two combinatorial conditions of relevance to our problem.

\begin{definition}\label{def:suitability}
  Let $k\le \ell$ be natural numbers and write $\psi=\lceil \ell/k\rceil$. We say that a pair $(\m,\n)\in \bN^k\times \bN^\ell$ is \emph{weakly suitable} if there exist $\bdelta \in \{0,1,2\}^{k}$ and $\bepsilon\in \{0,1,2\}^{\ell}$ such that
\begin{equation}\label{eq:weakly-suitable} \m-\bdelta\ge (\psi^k),\;%
 \n-\bepsilon\in \bN^\ell \; \text{ and } \; |\m-\bdelta|=|\n-\bepsilon|.
 \end{equation}
 If in addition the inequality \begin{equation}\label{ineq:stronglysuitable}
\sum_{j=1}^\psi(\n-\boldsymbol{\varepsilon})^*_j\ge k\psi
\end{equation}
holds, then we say that the pair $(\m,\n)$ is \emph{strongly suitable}.
We say that a pair $(\bdelta,\bepsilon)$
\emph{illustrates} the weak or strong suitability of $(\m,\n)$, as appropriate, when the corresponding conditions above hold.
\end{definition}

\begin{theorem}\label{thm:main}
Let $(\m,\n)\in\bN^k \times \bN^\ell$ where $k\le\ell$.
The following statements satisfy the implications $\text{(i)}\implies\text{(ii)}\implies\text{(iii)}\implies\text{(iv)}$:
\begin{enumerate}[(i)]
\item $(\m,\n)$ is strongly suitable;
\item $(\m,\n)$ is always join-orthogonalisable;
\item $(\m,\n)$ is weakly suitable;
\item $(\m,\n)$ is sometimes join-orthogonalisable.
\end{enumerate}
\end{theorem}

\begin{proof}%
  Let $\psi=\lceil \ell/k\rceil \ge1$. Then $\ell\le k\psi$.
  Suppose first that $(\m,\n)$ is a strongly suitable pair and that $(\bdelta,\bepsilon)$ illustrates the strong suitability of $(\m,\n)$.
  By~\eqref{ineq:stronglysuitable},
$$\ell\le k\psi\leq \sum_{j=1}^\psi(\n-\boldsymbol{\varepsilon})^*_j=|\bfd^*|=|\bfd|,$$
where $\bfd:=(\min\{\psi,n_j-\epsilon_j\})_{j\in [\ell]}\in \bN^\ell$.  Therefore we can decrease the entries of~$\bfd$ to obtain $\bfc\in {\bN}^{\ell}$ with $|\bfc|=k\psi$ such that $\bfc\le\bfd\le (\psi^\ell)$ and $\bfc\le \bfd\le{\n-\bepsilon}$. %
By~\eqref{eq:weakly-suitable} we see that
$\bfp:=\m-\bdelta-(\psi^k)$ and $\bfq:=\n-\bepsilon-\bfc$ have non-negative entries and satisfy \[|\bfp|=|\m-\bdelta|-k\psi = |\n-\bepsilon|-|\bfc|=|\bfq|.\] By Lemmas~\ref{lem:new1} and~\ref{lem:new2}, $((\psi^k)+\bfp,\bfc+\bfq)=(\m-\bdelta,\n-\bepsilon)\in \widetilde{\LL}(k,\ell)$.
By Remark~\ref{remark:D-sets-and-AJO}, $(\m,\n)$ is always join-orthogonalisable.

Suppose now that $(\m,\n)$ is always join-orthogonalisable.
By Remark~\ref{remark:D-AJO}, we have $(\m,\n)\in \LL(k,\ell)$ and so there exist
$\bdelta \in \{0,1,2\}^{k}$, $\bepsilon \in \{0,1,2\}^{\ell}$ so that $(\m-\bdelta, \n-\bepsilon) \in \widetilde{\LL}(k,\ell)$. It follows by Lemma~\ref{lem:necessary-condition-on-m} that $(\m,\n)$ is weakly suitable.

 Finally, suppose that $(\m,\n)$ is weakly suitable, and let $(\bdelta,\boldsymbol{\varepsilon})$ illustrate this.
To show that $(\m,\n)$ is sometimes join-orthogonalisable, it suffices to check that~\eqref{eq:weakly-suitable}
     implies that one of the conditions (i), (ii) and~(iii) in Theorem~\ref{thm:join-union-complete-graphs-2} holds.
     First note that $\m-\bdelta\ge (\psi^k)$ implies  that $|\m-\bdelta|= k\psi \ge\ell$. Hence we have $|\m|\ge |\m-\bdelta|\ge \ell$, and $|\n|\ge |\n-\bepsilon|=|\m-\bdelta|\ge \ell$. So if $\iota(\m)=0$, then either (i) or (iii) holds. If $\iota(\m)\ne 0$, then since $\m\ge \m-\bdelta\ge (\psi^k)$, we must have $\psi=1$, so $\ell=k$. Hence $|\m|+\iota(\m)=\sum_{i\in [k]:m_i\ge 2}m_i + 2\iota(\m)\ge 2k=k+\ell$, so (ii) holds.
\end{proof}

Determining whether a given pair is weakly or strongly suitable using the definitions of these properties can be inconvenient, as it involves a search over various vectors $\bdelta$ and $\bepsilon$. Our next goal (Proposition~\ref{prop:weakly-suitable-without-delta-epsilon}) is to find characterisations of these properties which do not require such a search.

The next technical lemma shows that among pairs $(\bdelta,\boldsymbol{\varepsilon})$ that could illustrate weak or strong suitability for $(\m,\n)$, it is enough to consider the pairs where either $\bdelta=(0^k)$ or $\boldsymbol{\varepsilon}=(0^\ell)$, depending on which of $|\m|$ and $|\n|$ is larger.

\begin{lemma}\label{lem:reduction}
  Let $k\le \ell$ be positive integers, $\psi=\lceil \ell/k \rceil$, $\m\in \bN^k$ and $\n\in \bN^\ell$.
  \begin{enumerate}[(i)]
  \item If $|\m|\ge |\n|$, then the following are equivalent:
    \begin{enumerate}
    \item $\m$ and $\n$ are weakly suitable (strongly suitable, respectively);
    \item there exists $\bdelta\in \{0,1,2\}^k$ so that
      the pair $(\bdelta,(0^\ell))$
      illustrates weak suitability (strong suitability, respectively) of $(\m,\n)$.
    \end{enumerate}
  \item If $|\m|\le |\n|$, then the following are equivalent:
    \begin{enumerate}
    \item $\m$ and $\n$ are weakly suitable (strongly suitable, respectively);
    \item there exists $\bepsilon\in \{0,1,2\}^\ell$ so that the pair
      $((0^k),\bepsilon)$
      illustrates weak suitability (strong suitability, respectively) of $(\m,\n)$.
    \end{enumerate}
    \item If $|\m|= |\n|$, then the following are equivalent:
    \begin{enumerate}
    \item $\m$ and $\n$ are weakly suitable (strongly suitable, respectively);
    \item the pair
      $((0^k),(0^{\ell}))$
      illustrates weak suitability (strong suitability, respectively) of $(\m,\n)$.
    \end{enumerate}
    \end{enumerate}
\end{lemma}
\begin{proof}
    \begin{enumerate}[(i)]
    \item By definition, (b) implies (a). Suppose (a) and let a pair $(\boldsymbol{\delta'},\boldsymbol{\varepsilon'})$ illustrate weak suitability (strong suitability, respectively) for $(\m,\n)$. In particular %
    $|\boldsymbol{\delta'}|=|\m|-|\n|+|\boldsymbol{\epsilon'}|\ge |\m|-|\n|$.
    Hence, we can choose $\bdelta\in\{0,1,2\}^k$ with $|\boldsymbol\delta|=|\m|-|\n|=|\boldsymbol{\delta'}|-|\boldsymbol{\epsilon'}|$ and $\bdelta\le \boldsymbol{\delta'}$.
    Clearly, $\m-\bdelta \geq \m-\boldsymbol{\delta'}\geq (\psi^k)$ and $\n\ge\n-\boldsymbol{\epsilon'}\ge (1^\ell)$, and $\sum_1^\psi n_j^*\ge \sum_1^\psi (\n-\boldsymbol{\epsilon'})_j^*$.
    This implies that $(\bdelta,(0^\ell))$ illustrates the weak suitability (strong suitability, respectively) of $(\m,\n)$, and thus (b) follows. An identical argument establishes the ``moreover'' claim.

    \item Again, suppose (a) and let a pair $(\boldsymbol{\delta'},\boldsymbol{\varepsilon'})$ illustrate weak suitability for $(\m,\n)$. By choosing $\bepsilon\le \boldsymbol{\epsilon'}$ with $|\bepsilon|=|\n|-|\m|$, we can prove $((0^k),\bepsilon)$
      illustrates weak suitability of $(\m,\n)$ by a symmetric argument as in (a). Moreover,
    \begin{align*}\sum_{j\in [\psi]} (\n-\bepsilon)^*_j&=\sum_{j\in [\psi]} |\{i \colon (\n-\bepsilon)_i\geq j\}|\\&\ge \sum_{j\in [\psi]} |\{i \colon (\n-\boldsymbol{\epsilon'})_i\geq j\}|= \sum_{j\in [\psi]} (\n-\boldsymbol{\epsilon'})^*_j,\end{align*}
    which implies that whenever $(\boldsymbol{\delta'},\boldsymbol{\varepsilon'})$ illustrate strong suitability for $(\m,\n)$, so does the pair $((0^k),\bepsilon)$.
    \item clearly follows by (i) and (ii).\qedhere
\end{enumerate}
\end{proof}

The following lemma is elementary, and we omit its easy proof.
\begin{lemma}\label{lem:epsilons}
  Consider integers $\psi,n,\epsilon,e'$
  where $\epsilon\in \{0,1,2\}$ and \[e'=
    \begin{cases}
      2&n\ge \psi+2\\
      1&n=\psi+1\\
      0&n\le \psi.
    \end{cases}
  \]
Then \[\min\{n-\epsilon,\psi\}=\min\{n,\psi\}+\min\{\epsilon,e'\}-\epsilon.\]
\end{lemma}

\begin{proposition}\label{prop:weakly-suitable-without-delta-epsilon}

  Let $k\le \ell$ and $\psi=\lceil \ell/k\rceil$.
  \begin{enumerate}[(i)]
  \item A pair $(\m,\n)\in\bN^{k}\times \bN^{\ell}$ is weakly suitable if and only if $\m\geq (\psi^k)$ and 
\begin{equation}\label{eq:weak-suitability-*}
    |\n|-n_2^*-n_3^*\le |\m|\le |\n|+m_{\psi+1}^*+m_{\psi+2}^*.
\end{equation} 
Moreover,~\eqref{eq:weak-suitability-*} is equivalent to
\begin{equation}\label{eq:-weak-suitability-majorization}
    \n \preceq_{w} (3+|\m|-\ell,3^{\ell-1})  \text{ and } |\n|\geq |\m|-m_{\psi+1}^*
-m_{\psi+2}^*.
\end{equation}
\item A pair $(\m,\n)\in\bN^{k}\times \bN^{\ell}$ is strongly suitable if and only if it is weakly suitable and 
\begin{align}
    \sum_{j=1}^\psi n^*_j&\ge  k\psi+\max\{0,|\n|-|\m|-n^*_{\psi+1}-n^*_{\psi+2}\}.\label{eq:strong3}
\end{align}
\end{enumerate}
\end{proposition}
\begin{proof}
First suppose that $|\n|\leq |\m|$.  By Lemma~\ref{lem:reduction}, $(\m,\n)$ is weakly suitable if and only if there is $\bdelta\in \{0,1,2\}^k$ so that $(\bdelta,\mathbf{0})$ illustrates weak suitability:
\begin{equation}\label{eq:weak}
\m-\bdelta\geq (\psi^k), \; \n \geq (1^{\ell}), \;|\m-\bdelta|=|\n|.
\end{equation}
Observe that $\m-\bdelta\geq (\psi^k)$ is equivalent to $\m\ge(\psi^k)$ and $\bdelta\le \sum_{i:m_i\ge \psi+1} \mathbf{e}_i+\sum_{j:m_j\ge \psi+2}\mathbf{e}_j$, that $\big|\sum_{i:m_i\ge \psi+1} \mathbf{e}_i+\sum_{j:m_j\ge \psi+2}\mathbf{e}_j\big|=m_{\psi+1}^*+m_{\psi+2}^*$, and that $|\m-\bdelta|=|\n|$ is equivalent to $|\bdelta|=|\m|-|\n|$. We conclude that there exists $\bdelta\in \{0,1,2\}^k$ so that \eqref{eq:weak} holds, if and only if $\m\geq (\psi^k)$ and $|\m|-|\n|\leq m_{\psi+1}^*+m_{\psi+2}^*$. 
By Lemma~\ref{lem:reduction}, $(\m,\n)$ is strongly suitable if and only if there is $\bdelta\in \{0,1,2\}^k$ so that $(\bdelta,\mathbf{0})$ illustrates strong suitability. Hence,  $(\m,\n)$ is strongly suitable if and only if it is weakly suitable and \eqref{eq:strong3} holds. 

Now suppose that $|\n|>|\m|$. In this case $(\m,\n)$ is weakly suitable if and only if there is $\bepsilon\in \{0,1,2\}^\ell$ so that $(\mathbf{0},\bepsilon)$ illustrates this: 
\begin{equation*}\label{eq:weak2}
\m\geq (\psi^k), \; \n-\bepsilon \geq (1^{\ell}),\; |\m|=|\n-\bepsilon|.
\end{equation*}
Equivalently, 
\begin{equation}\label{eq:weak3}
\m\geq (\psi^k), \; \bepsilon\leq \mathbf{e}:=\sum_{i:n_i\ge 2} \mathbf{e}_i+\sum_{j:n_j\ge 3}\mathbf{e}_j, \;
|\bepsilon|=|\n|-|\m|.
\end{equation} Since $|\mathbf{e}|=n_{2}^*+n_{3}^*$, there exists $\bepsilon\in \{0,1,2\}^\ell$ so that \eqref{eq:weak3} holds for some $\bepsilon\in \{0,1,2\}^\ell$ if and only if $\m\geq (\psi^k)$ and $|\n|-|\m|\le n_{2}^*+n_{3}^*$. 

For strong suitability, we first use Lemma~\ref{lem:reduction} and our arguments above to assert: $(\m,\n)$ is strongly suitable if and only if there exists $\bepsilon\in \{0,1,2\}^\ell$ so that \eqref{eq:weak3} and \eqref{ineq:stronglysuitable} hold. Within the set of $\bepsilon$ that satisfy \eqref{eq:weak3}, we aim to identify those for which $\sum_{j=1}^\psi(\n-\boldsymbol{\varepsilon})^*_j$ attains its maximum value. To this end we 
define $\mathbf{e'}:=\sum_{i:\psi+1\leq n_i} \mathbf{e}_i+\sum_{j:\psi+2\leq n_j}\mathbf{e}_j$ and $\mathbf{e''}:=\sum_{i:\psi+1>n_i\ge 2} \mathbf{e}_i+\sum_{j:\psi+2>n_j\ge 3}\mathbf{e}_j$, so that $\mathbf{e}=\mathbf{e'}+\mathbf{e''}$ and $|\mathbf{e'}|=n_{\psi+1}^*+n_{\psi+2}^*$. Note that any $\bepsilon$ satisfying $\bepsilon\leq \mathbf{e}$ can be written as $\bepsilon=\boldsymbol{\epsilon'}+\boldsymbol{\epsilon''}$ with $\boldsymbol{\epsilon'}:=\min\{\bepsilon,\boldsymbol{e'}\}$ and $\boldsymbol{\epsilon''}=\bepsilon-\boldsymbol{\epsilon'}$.  
By Lemma~\ref{lem:epsilons}, we have
\begin{align*} 
\sum_{j=1}^\psi (\n-\bepsilon)_j^*&=\sum_{j=1}^\ell \min\{n_j-\epsilon_j,\psi\}=\sum_{j=1}^\ell (\min\{n_j,\psi\}+\min\{\epsilon,e'_j\}-\epsilon_j)\\&=\left(\sum_{j=1}^\psi n_j^*\right)-|\boldsymbol{\epsilon''}|.
\end{align*}
Hence this quantity is maximised when $|\boldsymbol{\epsilon''}|$ is as small as possible, i.e., when $\bepsilon$ satisfies
\begin{equation}\label{eq:epsilonchoice}
\text{$\bepsilon\le \boldsymbol{e'}$ if $|\n|-|\m|\le |\boldsymbol{e'}|$, and $\boldsymbol{e'}\le \bepsilon\le\boldsymbol{e}$ otherwise.}
\end{equation}
For such $\bepsilon$ we have
$\sum_{j=1}^\psi(\n-\boldsymbol{\varepsilon})^*_j=\sum_{j=1}^\psi n^*_j-\max\{0,|\n|-|\m|-n^*_{\psi+1}-n^*_{\psi+2}\}$. It follows that the existence of $\bepsilon\in \{0,1,2\}$ satisfying \eqref{eq:weak3} and \eqref{ineq:stronglysuitable} is equivalent to condition~\eqref{eq:strong3}. %

Finally, we prove that~\eqref{eq:weak-suitability-*} is equivalent to~\eqref{eq:-weak-suitability-majorization}. Without loss, we may assume that $\n$ is non-increasing. Since $n_1^*=\ell$, observe that $|\n|-n_2^*-n_3^*\le |\m|$ in~(\ref{corollary:k-divides-l-part-3}) is equivalent to \begin{equation}\label{eq:sum-n4*}
        \sum_{i\geq 4} n_i^*\leq |\m|-\ell.
    \end{equation}
If $\n$ satisfies~\eqref{eq:sum-n4*}, let $\bfp=(n_4^*,n_5^*,\dots)^*$. Then $|\bfp|\le |\m|-\ell$, and we have $\n=(3^{a},2^b,1^c)+\bfp$ where $a=n_3^*$, $b=n_2^*-n_3^*$ and $c=\ell-n_2^*$. In particular, since $\n\leq (3^{\ell})+\bfp$, it follows that $\n \preceq_{w} (3^{\ell})+(|\m|-\ell,0^{\ell-1})=(3+|\m|-\ell,3^{\ell-1})$ and so (\ref{corollary:k-divides-l-part-3}) implies (\ref{corollary:k-divides-l-part-4}).
   Conversely, let $\n \preceq_{w} \mathbf{u}:=(3+|\m|-\ell,3^{\ell-1})$. If $n_4^*=0$, then~\eqref{eq:sum-n4*} clearly holds. Otherwise,  let $j=n_4^*\ge1$, so that $n_j\geq 4$ and $n_{j+1}\leq 3$. Then %
    we have 
    \[
    3j+\sum_{i\ge 4}n_i^*=\sum_{i\in [j]}n_i\le \sum_{i\in [j]}u_i= 3j+|\m|-\ell
    \]
    which implies~\eqref{eq:sum-n4*}. Thus (\ref{corollary:k-divides-l-part-4}) implies (\ref{corollary:k-divides-l-part-3}).
\end{proof}

\section{Equivalence and non-equivalence}

In this section, we show that while the various implications in Theorems~\ref{thm:main:nec-suff-q=2} and~\ref{thm:main} cannot generally be reversed, they may be under certain additional hypotheses.

\begin{proposition}\label{prop:strict-inclusions}
  None of the five implications in Theorems~\ref{thm:main:nec-suff-q=2} and~\ref{thm:main} is an equivalence in general.
\end{proposition}
\begin{proof}
{\bf Theorem~\ref{thm:main}: (ii) does not imply (i), and Theorem~\ref{thm:main:nec-suff-q=2}: (ii) does not imply (i).} To see that the converse of the first implication in Theorem~\ref{thm:main} is not true,
  consider the pair
  $(\m,\n)=((2^3),(3,1^3))\in \bN^3\times \bN^4$. We have $|\m|=|\n|$, so it follows easily using Lemma~\ref{lem:reduction}(iii) that \eqref{ineq:stronglysuitable} is not fulfilled, and hence $(\m,\n)$ is not strongly suitable. On the other hand, for $\bfr=(3,1^3)$, the matrices
  \begin{equation}\label{eq:222,3111}
    V=\left(
      \begin{tabular}{ccc}
        1&1&1\\
        1&0&0\\
        0&1&0\\
        0&0&1
      \end{tabular}\right)
    \text{ and }
    W=
    \left(
      \begin{tabular}{ccccccccc}
        0&1&1&1\\
        1&0&0&0\\
        1&0&0&0\\
        1&0&0&0
      \end{tabular}
    \right)
  \end{equation}
  satisfy $V\in \A(\bfr,\m)$, $W\in \A(\bfr,\n)$ and $V\trans W\in \bN^{3\times 4}$, so $(\m,\n)\in \LL(3,4)\subset\widetilde{\LL}(3,4)$, hence $(\m,\n)$ is always join-orthogonalisable by Remark~\ref{remark:D-AJO}. Moreover, the same pair $(\m,\n)=((2^3),(3,1^3))$ does not fulfil the last inequality in Theorem~\ref{thm:main:nec-suff-q=2}(i) since in this case the left-hand side of the inequality is equal to $5$ and the right-hand side is equal to $6$, therefore Theorem~\ref{thm:main:nec-suff-q=2}(ii) does not imply (i).
  
\noindent{\bf Theorem~\ref{thm:main}: (iii) does not imply (ii).}  Observe (either by definition or using Proposition~\ref{prop:weakly-suitable-without-delta-epsilon})
  that the pair $(\m,\n)=((2^4),(3,1^5))\in \bN^4\times \bN^6$ is
  weakly suitable. 
  However,
  it is not always join-orthogonalisable.
  To see this, suppose to the
  contrary; then $(\m,\n)\in \LL(4,6)$. By
  Remark~\ref{remark:D-sets-and-AJO}(\ref{remark:D-vs-tilde(D)}), there exist vectors $\bdelta$ and $\bepsilon$ with non-negative entries such that $(\m-\bdelta,\n-\bepsilon)\in \widetilde{\LL}(4,6)$. Since $|\m|=|\n|$, by  Remark~\ref{remark:D-sets-and-AJO}(\ref{remark:D-implies-same-size}) we have $|\bepsilon|=|\bdelta|$. By
  Lemma~\ref{lem:necessary-condition-on-m} we have $\m=(2^4)\le\m-\bdelta$, so $\bdelta=\boldsymbol{0}$ and hence $\bepsilon=\boldsymbol{0}$, so $(\m,\n)\in \widetilde{\LL}(4,6)$.
    However, $\iota_{\max}(4,6)=4<\iota(\n)$, which contradicts Proposition~\ref{prop:necc2}.
For an example, see Figure~\ref{fig:SJO-but-not-AJO}. 

\begin{figure}[htb] 
 \tikzset{
    every node/.style={draw, circle, fill=white, inner sep=1pt}
    }
\begin{subfigure}{0.45\textwidth}  
  \centering
\begin{tikzpicture}[scale=2.5]
\foreach \x in {-1.5,-1,...,1.5,2}{
  \foreach \y in {-1.5,-1,...,1,2}{
  \draw[color=gray] (0,\x)--(2,\y); 
    }
    \draw[color=gray] (0,\x)--(2.25,1.5); 
    }
\foreach \x in {-1.5,-0.5,0.5,1.5}{
            \draw[thick] (0,\x)--(0,\x+0.5);
             \node at (0,\x) { };
             \node at (0,\x+0.5) { };
    }
\foreach \x in {-1.5,-1,-0.5,0,0.5}{
             \node at (2,\x) { };
    }
 \node (1)  at (2,1) {}; 
 \node (2)  at (2.25,1.5) {}; 
 \node (3)  at (2,2) {}; 
\draw[thick] (1) -- (2) --(3);
\end{tikzpicture}
    \caption{$4P_2 \vee (P_3\cup 5P_1)$}
     \label{fig:not-AJO2}
    \end{subfigure}
\begin{subfigure}{0.45\textwidth}  
  \centering
\begin{tikzpicture}[scale=2.5]
\foreach \x in {-1.5,-1,...,1.5,2}{
  \foreach \y in {-1.5,-1,...,1,2}{
  \draw[color=gray] (0,\x)--(2,\y); 
    }
    \draw[color=gray] (0,\x)--(2.25,1.5); 
    }
\foreach \x in {-1.5,-0.5,0.5,1.5}{
            \draw[thick] (0,\x)--(0,\x+0.5);
             \node at (0,\x) { };
             \node at (0,\x+0.5) { };
    }
\foreach \x in {-1.5,-1,-0.5,0,0.5}{
             \node at (2,\x) { };
    }
 \node (1)  at (2,1) {}; 
 \node (2)  at (2.25,1.5) {}; 
 \node (3)  at (2,2) {}; 
\draw[thick] (1) -- (2) --(3);
\draw[thick] (3)--(1);
\end{tikzpicture}
\caption{$4P_2 \vee (K_3\cup 5P_1)$}  \label{fig:SJO1}
\end{subfigure}
\caption{The pair $(\m,\n)=((2^4),(3,1^5))$ is sometimes but not always join-orthogonalisable (see the proof of Proposition~\ref{prop:strict-inclusions}). This figure shows two graphs, both sized by $(\m,\n)$ and differing by only one edge, where the graph in (b) has $q=2$, but the graph in (a) does not. This follows immediately from Remark~\ref{remark:SJO-K-AJO-P}.}\label{fig:SJO-but-not-AJO}
\end{figure}

\noindent{\bf Theorem~\ref{thm:main}: (iv) does not imply (iii).} It is easy to see that $(\m,\n)=((4,1),(1,1))$ is not weakly suitable, but it satisfies property~(ii) in Theorem~\ref{thm:join-union-complete-graphs-2}, so $(\m,\n)$ is sometimes join-orthogonalisable.

\noindent{\bf Theorem~\ref{thm:main:nec-suff-q=2}: (iv) does not imply (iii).} It is straightforward to check that $(\m,\n):=((2^5),(3,2,1^5))\in \bN^5\times \bN^7$ satisfies (iv); we claim that this pair is not always join-orthogonalisable. To see this, suppose to the
  contrary that $(\m,\n)\in \LL(5,7)$. By a similar argument as above,
  Remark~\ref{remark:D-sets-and-AJO}(\ref{remark:D-vs-tilde(D)}),(\ref{remark:D-implies-same-size}) and 
  Lemma~\ref{lem:necessary-condition-on-m} imply that $(\m,\n)\in \widetilde{\LL}(5,7)$, so there are $0$-$1$ matrices $V$ and $W$ with equal row-sums such that the column-sums of $V$ and
  $W$ are $\m=(2^5)$ and $\n=(3,2,1^5)$, respectively, and
  $V\trans W\in \bN^{5\times 7}$. Five columns of $W$ have sum $1$; we call such columns \emph{elementary}. By
  permuting rows of $V$ and $W$, we may assume that
  $\mathbf{e}_1\trans$ is a column of $W$, and since $V\trans W$ has no zero entry, this implies that the first row of $V$ is full of ones.
  Hence, the first row-sum of $W$ is~$5$. Since there are only two columns of $W$ with a sum of more than one, at least three columns of $W$ must be equal to $\mathbf{e}_1\trans$. If any of the elementary columns of $W$ were different from $\mathbf{e}_1\trans$, then by the same argument, there would be at least three such columns. Then $W$ would contain at least six elementary columns, whereas we know it has exactly five. Therefore, $W$ has five columns which are all equal to $\mathbf{e}_1\trans$ and the first row of $V$ is full of ones. 
  Let $V'$ be the $0$-$1$ matrix with column sums $(1,1,1,1,1)$ obtained by deleting the first row of $V$, and let $W'$ be the $0$-$1$ matrix with column sums $(3,2)$ obtained by deleting the first row of $W$ and the five columns equal to $\mathbf{e}_1\trans$. Since $V'$ and $W'$ are compatible, it follows that $((3,2),(1^5))\in \widetilde{\LL}(2,5)$, which contradicts Lemma~\ref{lem:necessary-condition-on-m}.
\end{proof}

We now investigate the influence the parameter $\iota(\n)$ can have on whether or not a pair $(\m,\n)$ is always join-orthogonalisable. 
Roughly speaking, our next result (which completes the proof of Theorem~\ref{thm:main:nec-suff-q=2}) shows that always join-orthogonalisablity of $(\m,\n)$ implies that $\iota(\n)$ is not too large. %
Then in Proposition~\ref{prop:omega=1} we show that in some cases, a converse result holds. %

\begin{corollary}\label{cor:paths-q>2}
    If $k <\ell$
    and $(\m,\n)\in \bN^k\times \bN^\ell$ is always join-orthogonalisable, 
    then $\iota(\n)\le \iota_{\max}(k,\ell)$.
\end{corollary}
\begin{proof}
  By Remark~\ref{remark:D-sets-and-AJO}(\ref{remark:D-vs-tilde(D)}), there exist
  $\bdelta\in \{0,1,2\}^k$ and
  $\bepsilon\in \{0,1,2\}^{\ell}$ such that we have
  $(\m-\bdelta,\n-\bepsilon)\in
  \widetilde{\LL}(k,\ell)$. Hence,
  $\iota(\n)\le \iota(\n- \bepsilon)\le
  \iota_{\max}(k,\ell)$, by Proposition~\ref{prop:necc2}.
\end{proof}

\begin{proposition}\label{prop:omega=1}
  Let $k\le \ell$, $\psi=\lceil \ell/k\rceil$ and suppose $(\m,\n)\in \bN^k\times \bN^\ell$ is
  weakly suitable, illustrated by
  $(\bdelta,\bepsilon)$. If $\iota(\n-\bepsilon)\le 2\ell-\psi k$, then $(\m,\n)$ is strongly suitable and
  always join-orthogonalisable.
\end{proposition}
\begin{proof}
  By Theorem~\ref{thm:main}, we need only establish strong suitability.
  Conditions~(\ref{eq:weakly-suitable})
  hold, and we need to show that the
  inequality~(\ref{ineq:stronglysuitable}) is satisfied.

  If $\psi=1$ then $k=\ell$, so $k\psi=\ell= (\n-\bepsilon)_1^*$, which implies~(\ref{ineq:stronglysuitable}). 
  So we may assume that $\psi\ge2$.
  Now
  \[\sum_{j=1}^\psi (\n-\bepsilon)_j^*\ge (\n-\bepsilon)_1^*+(\n-\bepsilon)_2^*=\ell+(\ell-\iota(\n-\bepsilon))=2\ell-\iota(\n-\bepsilon)\ge \psi k.\]
  The result follows.
\end{proof}

\begin{corollary}\label{cor:k-divisible-corollary}
  If $k$ divides either $\ell$ or $\ell+1$, then conditions (i), (ii) and (iii) in Theorem~\ref{thm:main} are equivalent and conditions (i), (ii), (iii) and (iv) in Theorem~\ref{thm:main:nec-suff-q=2} are equivalent.
\end{corollary}
\begin{proof}
    Let $\psi=\lceil\ell/k\rceil$.
    If $k$ divides $\ell$, then $\iota(\n')\le \ell = 2\ell-\psi k$ for any $\n'\in \bN^\ell$.

    If $k>1$ and $k$ divides $\ell+1$, then $\psi k = \ell+1$. Suppose $(\bdelta,\bepsilon)$ illustrates the weak suitability of $(\m,\n)\in \bN^k\times \bN^\ell$. Then $\ell+1=\psi k=|(\psi^k)|\le |\m-\bdelta|=|\n-\bepsilon|$, so $\n-\bepsilon\ne (1^\ell)$, so $\iota(\n-\bepsilon)\le \ell-1=2\ell-\psi k$.

    The result now follows by Proposition~\ref{prop:omega=1} and Proposition~\ref{prop:weakly-suitable-without-delta-epsilon}.
\end{proof}

 Note that Corollary~\ref{cor:paths-q>2} is independent of Theorem~\ref{thm:main}. For example, the pair $(\m,\n)=((2^4),(3,1^5))\in \bN^4\times \bN^6$ considered in the proof of Proposition~\ref{prop:strict-inclusions} is not always join-orthogonalisable, and this follows immediately from Corollary~\ref{cor:paths-q>2} since $\iota_{\max}(4,6)=4<\iota(\n)=5$. Since $(\m,\n)$ is weakly suitable, we cannot draw the same conclusion from Theorem~\ref{thm:main}.
  This example may be generalised as follows.

\begin{example}
  Let $4\le k<\ell$ and suppose $\ell\equiv 2\bmod k$, so that
  $\ell=(\psi-1)k+2$ where $\psi=\lceil \ell/k\rceil$. We have $\iota_{\max}(k,\ell)\ge (\psi-1)k=\ell-2$. By
  Corollary~\ref{cor:paths-q>2}, $(\m,\n)\in \bN^k\times \bN^\ell$ is
  not always join-orthogonalisable if $\iota(\n)\ge \ell-1$. By
  Remark~\ref{remark:SJO-K-AJO-P}, we have
  $$q\left(\bigcup_{i\in [k]}P_{m_i}\vee \left(P_{n_1} \cup (\ell-1) P_1\right)\right)\geq 3$$
  for any $\m\in \bN^k$ and $n_1\in \bN$.
\end{example}

\section{Orthogonalisability of joins of graphs}
In the final section of this work, we return to our main motivation of determining conditions for a pair of 
graphs $G$ and $H$ to have $q(G \vee H)=2$. Results that follow from Section \ref{sec:suitability} are already summarised in Theorem~\ref{thm:main:nec-suff-q=2}. The next result is a direct consequence of  Theorem~\ref{thm:main:nec-suff-q=2}, Proposition~\ref{prop:weakly-suitable-without-delta-epsilon} and Corollary~\ref{cor:k-divisible-corollary}. 

\begin{corollary}\label{cor:k-divides-l-or-l+1}
Suppose $k$ divides either $\ell$ or $\ell+1$. Let $(\m,\n)\in \bN^k\times \bN^\ell$ and $\psi=\lceil \ell/k\rceil$. The following are equivalent: 
\begin{enumerate}[(i)]
\item \label{corollary:k-divides-l-part-1} $q(G\vee H)=2$ for all graphs $G,H$ which are sized by $(\m,\n)$;
\item \label{corollary:k-divides-l-part-2} $q(G\vee H)=2$ for $G=\bigcup_{i=1}^k P_{m_i}$ and $H=\bigcup_{j=1}^\ell P_{n_j}$; 
\item \label{corollary:k-divides-l-part-3} $\m\ge (\psi^k)$ and $|\n|-n_2^*-n_3^*\le |\m|\le |\n|+m_{\psi+1}^*+m_{\psi+2}^*$.
\item \label{corollary:k-divides-l-part-4} $\m\ge (\psi^k)$ and $\n \preceq_{w} (3+|\m|-\ell,3^{\ell-1})$  and $|\n|\geq |\m|-m_{\psi+1}^*
-m_{\psi+2}^*$.
\end{enumerate}
\end{corollary}

Note that Corollary~\ref{cor:k-divides-l-or-l+1} completely characterises the always join-orthogonalisable pairs $(\m,\n)$, when $k\leq 2$, i.e., $\m\in \bN$ or $\m\in \bN^2$. In the next example, we explain the case $k=1$.

\begin{example}
    Bjorkman~et al.~\cite[Example~4.5]{18-Bjorkman-q} proved that $q(P_m\vee P_1)=\lceil\frac{m+1}{2}\rceil$ for $m\ge2$. In particular, $q(P_m\vee P_1)=2$ if and only if  $1\le m\le 3$. If we set $k=1$ in Corollary~\ref{cor:k-divides-l-or-l+1}, the latter fact generalises as follows: If $m,\ell\in \bN$ and $\n\in\bN^{\ell}$, then
     $$q\left(G\vee \bigcup_{i\in [\ell]} H_{i} \right)=2$$ for all connected graphs $G$ of order $m$ and all graphs $H$ which are sized by $\n$, if and only if
     $$|\n|-n_2^*-n_3^*\leq m\leq |\n|+2,$$
     or, equivalently
     $$\ell \leq m \leq |\n|+2 \text{ and } \n \preceq_w (3+m-\ell,3^{\ell-1}).$$
      In particular, when $\n\leq (3^{\ell})$, we have  $q\left(P_{m}\vee \bigcup_{i\in [\ell]} P_{n_i} \right)=2$ if and only if
     $\ell\le m\le |\n|+2$. All such graphs with $|\n|\leq 3$ are illustrated in Table~\ref{tab:Pm-Pn-q=2}. Therefore 
     $q\left(P_{m}\vee (\ell P_1) \right)=2$ if and only if
     $\ell\le m\le \ell+2$. 

    \tikzset{
    every node/.style={draw, circle, fill=white, inner sep=1pt}
    }
\begin{table}
    \centering
    \begin{tabular}{c||c|c|c|}
         & $|\n|=1$ & $|\n|=2$ & $|\n|=3$\\
         \hline\hline
       $m=1$  &  \begin{tikzpicture}
\centering
  \draw[color=gray] (0,0)--(2,0); 
  \node at (0,0) { };
 \node  at (2,0) {}; 
\end{tikzpicture}
&  \begin{tikzpicture}
\centering
\foreach \x in {0}{
  \foreach \y in {-0.5,0.5}{
  \draw[color=gray] (0,\x)--(2,\y); 
    }
    }
  \node at (0,0) { };
\foreach \x in {-0.5,0.5}{
             \node at (2,\x) { };
    }
 \node (1)  at (2,-0.5) {}; 
 \node (2)  at (2,0.5) {}; 
 \draw[thick] (1)--(2);
\end{tikzpicture}
&     \begin{tikzpicture}
\centering
\foreach \x in {0}{
  \foreach \y in {-1,0,1}{
  \draw[color=gray] (0,\x)--(2,\y); 
    }
    }
      \node at (0,0) { };
\foreach \x in {-1,0,1}{
             \node at (2,\x) { };
    }
 \node (1)  at (2,-1) {}; 
 \node (2)  at (2,0) {}; 
 \node (3)  at (2,1) {}; 
\draw[thick] (1) -- (2) --(3);
\end{tikzpicture}
\\
       \hline
       $m=2$   & \begin{tikzpicture}
\centering
\foreach \x in {0}{
  \foreach \y in {-0.5,0.5}{
  \draw[color=gray] (2,\x)--(0,\y); 
    }
    }
  \node at (2,0) { };
\foreach \x in {-0.5,0.5}{
             \node at (0,\x) { };
    }
 \node (1)  at (0,-0.5) {}; 
 \node (2)  at (0,0.5) {}; 
 \draw[thick] (1)--(2);
\end{tikzpicture}
&  \begin{tikzpicture}
\centering
\foreach \x in {-0.5,0.5}{
  \foreach \y in {-0.5,0.5}{
  \draw[color=gray] (0,\x)--(2,\y); 
    }
    }
    \foreach \x in {-0.5}{
            \draw[thick] (0,\x)--(0,\x+1);
             \node at (0,\x) { };
    }
\foreach \x in {-0.5,0.5}{
             \node at (2,\x) { };
    }
 \node (0)  at (0,0.5) {}; 
\end{tikzpicture}
 \begin{tikzpicture}
\centering
\foreach \x in {-0.5,0.5}{
  \foreach \y in {-0.5,0.5}{
  \draw[color=gray] (0,\x)--(2,\y); 
    }
    }
    \foreach \x in {-0.5}{
            \draw[thick] (0,\x)--(0,\x+1);
             \node at (0,\x) { };
    }
\foreach \x in {-0.5,0.5}{
             \node at (2,\x) { };
    }
 \node (0)  at (0,0.5) {}; 
 \node (1)  at (2,-0.5) {}; 
 \node (2)  at (2,0.5) {}; 
 \draw[thick] (1)--(2);
\end{tikzpicture}& 
          \begin{tikzpicture}
\centering
\foreach \x in {-0.5,0.5}{
  \foreach \y in {-1,0,1}{
  \draw[color=gray] (0,\x)--(2,\y); 
    }
    }
    \foreach \x in {-0.5}{
            \draw[thick] (0,\x)--(0,\x+1);
             \node at (0,\x) { };
    }
\foreach \x in {-1,0,1}{
             \node at (2,\x) { };
    }
 \node (0)  at (0,0.5) {}; 
 \node (1)  at (2,-1) {}; 
 \node (2)  at (2,0) {}; 
 \node (3)  at (2,1) {}; 
\draw[thick]  (2) --(3);
\end{tikzpicture}
   \begin{tikzpicture}
\centering
\foreach \x in {-0.5,0.5}{
  \foreach \y in {-1,0,1}{
  \draw[color=gray] (0,\x)--(2,\y); 
    }
    }
    \foreach \x in {-0.5}{
            \draw[thick] (0,\x)--(0,\x+1);
             \node at (0,\x) { };
    }
\foreach \x in {-1,0,1}{
             \node at (2,\x) { };
    }
 \node (0)  at (0,0.5) {}; 
 \node (1)  at (2,-1) {}; 
 \node (2)  at (2,0) {}; 
 \node (3)  at (2,1) {}; 
\draw[thick]  (1)--(2) --(3);
\end{tikzpicture}
\\
       \hline
       $m=3$   & 
        \begin{tikzpicture}
\centering
\foreach \x in {0}{
  \foreach \y in {-1,0,1}{
  \draw[color=gray] (2,\x)--(0,\y); 
    }
    }
      \node at (2,0) { };
\foreach \x in {-1,0,1}{
             \node at (0,\x) { };
    }
 \node (1)  at (0,-1) {}; 
 \node (2)  at (0,0) {}; 
 \node (3)  at (0,1) {}; 
\draw[thick] (1) -- (2) --(3);
\end{tikzpicture}
& 
        \begin{tikzpicture}
\centering
\foreach \x in {-1,0,1}{
  \foreach \y in {-0.5,0.5}{
  \draw[color=gray] (0,\x)--(2,\y); 
    }
    }
 \foreach \x in {-1,0}{
            \draw[thick] (0,\x)--(0,\x+1);
             \node at (0,\x) { };
    }
\foreach \x in {-0.5,0.5}{
             \node at (2,\x) { };
    }
 \node (0)  at (0,1) {}; 
\end{tikzpicture}
 \begin{tikzpicture}
\centering
\foreach \x in {-1,0,1}{
  \foreach \y in {-0.5,0.5}{
  \draw[color=gray] (0,\x)--(2,\y); 
    }
    }
 \foreach \x in {-1,0}{
            \draw[thick] (0,\x)--(0,\x+1);
             \node at (0,\x) { };
    }
\foreach \x in {-0.5,0.5}{
             \node at (2,\x) { };
    }
 \node (0)  at (0,1) {}; 
 \node (1)  at (2,-0.5) {}; 
 \node (2)  at (2,0.5) {}; 
 \draw[thick] (1)--(2);
\end{tikzpicture}& 
       \begin{tikzpicture}
\centering
\foreach \x in {-1,0,1}{
  \foreach \y in {-1,0,1}{
  \draw[color=gray] (0,\x)--(2,\y); 
    }
    }
 \foreach \x in {-1,0}{
            \draw[thick] (0,\x)--(0,\x+1);
             \node at (0,\x) { };
    }
\foreach \x in {-1,0,1}{
             \node at (2,\x) { };
    }
 \node (0)  at (0,1) {}; 
\end{tikzpicture}
\begin{tikzpicture}
\centering
\foreach \x in {-1,0,1}{
  \foreach \y in {-1,0,1}{
  \draw[color=gray] (0,\x)--(2,\y); 
    }
    }
 \foreach \x in {-1,0}{
            \draw[thick] (0,\x)--(0,\x+1);
             \node at (0,\x) { };
    }
\foreach \x in {-1,0,1}{
             \node at (2,\x) { };
    }
 \node (0)  at (0,1) {}; 
 \node (1)  at (2,-1) {}; 
 \node (2)  at (2,0) {}; 
 \node (3)  at (2,1) {}; 
\draw[thick] (2) --(3);
\end{tikzpicture}
\begin{tikzpicture}
\centering
\foreach \x in {-1,0,1}{
  \foreach \y in {-1,0,1}{
  \draw[color=gray] (0,\x)--(2,\y); 
    }
    }
 \foreach \x in {-1,0}{
            \draw[thick] (0,\x)--(0,\x+1);
             \node at (0,\x) { };
    }
\foreach \x in {-1,0,1}{
             \node at (2,\x) { };
    }
 \node (0)  at (0,1) {}; 
 \node (1)  at (2,-1) {}; 
 \node (2)  at (2,0) {}; 
 \node (3)  at (2,1) {}; 
\draw[thick] (1) -- (2) --(3);
\end{tikzpicture}
\\
       \hline
       $m=4$   &  & 
        \begin{tikzpicture}
\centering
\foreach \x in {-1.5,-0.5,0.5,1.5}{
  \foreach \y in {-0.5,0.5}{
  \draw[color=gray] (0,\x)--(2,\y); 
    }
    }
    \foreach \x in {-1.5,-0.5,0.5}{
            \draw[thick] (0,\x)--(0,\x+1);
             \node at (0,\x) { };
    }
\foreach \x in {-0.5,0.5}{
             \node at (2,\x) { };
    }
 \node (0)  at (0,1.5) {}; 
\end{tikzpicture}       
        \begin{tikzpicture}
\centering
\foreach \x in {-1.5,-0.5,0.5,1.5}{
  \foreach \y in {-0.5,0.5}{
  \draw[color=gray] (0,\x)--(2,\y); 
    }
    }
    \foreach \x in {-1.5,-0.5,0.5}{
            \draw[thick] (0,\x)--(0,\x+1);
             \node at (0,\x) { };
    }
\foreach \x in {-0.5,0.5}{
             \node at (2,\x) { };
    }
 \node (0)  at (0,1.5) {}; 
 \node (1)  at (2,-0.5) {}; 
 \node (2)  at (2,0.5) {}; 
 \draw[thick] (1)--(2);
\end{tikzpicture} 
&
         \begin{tikzpicture}
\centering
\foreach \x in {-1.5,-0.5,0.5,1.5}{
  \foreach \y in {-1,0,1}{
  \draw[color=gray] (0,\x)--(2,\y); 
    }
    }
    \foreach \x in {-1.5,-0.5,0.5}{
            \draw[thick] (0,\x)--(0,\x+1);
             \node at (0,\x) { };
    }
\foreach \x in {-1,0,1}{
             \node at (2,\x) { };
    }
 \node (0)  at (0,1.5) {}; 
 \node (1)  at (2,-1) {}; 
 \node (2)  at (2,0) {}; 
 \node (3)  at (2,1) {}; 
\end{tikzpicture}
  \begin{tikzpicture}
\centering
\foreach \x in {-1.5,-0.5,0.5,1.5}{
  \foreach \y in {-1,0,1}{
  \draw[color=gray] (0,\x)--(2,\y); 
    }
    }
    \foreach \x in {-1.5,-0.5,0.5}{
            \draw[thick] (0,\x)--(0,\x+1);
             \node at (0,\x) { };
    }
\foreach \x in {-1,0,1}{
             \node at (2,\x) { };
    }
 \node (0)  at (0,1.5) {}; 
 \node (1)  at (2,-1) {}; 
 \node (2)  at (2,0) {}; 
 \node (3)  at (2,1) {}; 
\draw[thick]  (2) --(3);
\end{tikzpicture}
  \begin{tikzpicture}
\centering
\foreach \x in {-1.5,-0.5,0.5,1.5}{
  \foreach \y in {-1,0,1}{
  \draw[color=gray] (0,\x)--(2,\y); 
    }
    }
    \foreach \x in {-1.5,-0.5,0.5}{
            \draw[thick] (0,\x)--(0,\x+1);
             \node at (0,\x) { };
    }
\foreach \x in {-1,0,1}{
             \node at (2,\x) { };
    }
 \node (0)  at (0,1.5) {}; 
 \node (1)  at (2,-1) {}; 
 \node (2)  at (2,0) {}; 
 \node (3)  at (2,1) {}; 
\draw[thick] (1) -- (2) --(3);
\end{tikzpicture}
\\
       \hline
       $m=5$   &  &  &     
       \begin{tikzpicture}
       
\centering
\foreach \x in {-2,-1,0,1,2}{
  \foreach \y in {-1,0,1}{
  \draw[color=gray] (0,\x)--(2,\y); 
    }
    }
    \foreach \x in {-2,-1,0,1}{
            \draw[thick] (0,\x)--(0,\x+1);
             \node at (0,\x) { };
    }
\foreach \x in {-1,0,1}{
             \node at (2,\x) { };
    }
 \node (0)  at (0,2) {}; 
 \node (1)  at (2,-1) {}; 
 \node (2)  at (2,0) {}; 
 \node (3)  at (2,1) {}; 
\end{tikzpicture}
\begin{tikzpicture}
       
\centering
\foreach \x in {-2,-1,0,1,2}{
  \foreach \y in {-1,0,1}{
  \draw[color=gray] (0,\x)--(2,\y); 
    }
    }
    \foreach \x in {-2,-1,0,1}{
            \draw[thick] (0,\x)--(0,\x+1);
             \node at (0,\x) { };
    }
\foreach \x in {-1,0,1}{
             \node at (2,\x) { };
    }
 \node (0)  at (0,2) {}; 
 \node (1)  at (2,-1) {}; 
 \node (2)  at (2,0) {}; 
 \node (3)  at (2,1) {}; 
\draw[thick]  (2) --(3);
\end{tikzpicture}
\begin{tikzpicture}
       
\centering
\foreach \x in {-2,-1,0,1,2}{
  \foreach \y in {-1,0,1}{
  \draw[color=gray] (0,\x)--(2,\y); 
    }
    }
    \foreach \x in {-2,-1,0,1}{
            \draw[thick] (0,\x)--(0,\x+1);
             \node at (0,\x) { };
    }
\foreach \x in {-1,0,1}{
             \node at (2,\x) { };
    }
 \node (0)  at (0,2) {}; 
 \node (1)  at (2,-1) {}; 
 \node (2)  at (2,0) {}; 
 \node (3)  at (2,1) {}; 
\draw[thick] (1) -- (2) --(3);

\end{tikzpicture}\\
\hline
    \end{tabular}
    \caption{All graphs of the form $P_m \vee \cup_{j=1}^{\ell} P_{n_j}$, having  $|\n|\leq 3$ and  $q=2$. Note that there are no such graphs for $m \geq 6$. }
    \label{tab:Pm-Pn-q=2}
\end{table}

\end{example}

\begin{example}
In~\cite[Theorem~3.4]{levene2023distinct} it was shown that if $G_i$ and $H_i$ are connected graphs with $\big||G_i|-|H_i|\big|\leq 2$, $i\in [k]$, then
$$q\left(\bigcup_{i\in [k]}G_{i}\vee \bigcup_{i\in [k]}H_{i}\right)=2.$$ Taking $\ell=k$ in Corollary~\ref{cor:k-divides-l-or-l+1},
we obtain a stronger result:
$q(G\vee H)=2$ for all pairs of graphs $(G,H)$ which are sized by $(\m,\n)\in \bN^k\times \bN^k$
if and only if
\begin{equation}\label{eq:k=l-condition}|\n|-n_2^*-n_3^*\le |\m|\le |\n|+m_2^*+m_3^*.\end{equation}
If $|m_i-n_i|\le 2$ for each $i\in [k]$, then~\eqref{eq:k=l-condition} holds, because
\begin{align*}
|\m|-|\n|=\sum_{i\in [k]}m_i-n_i&\le \sum_{i:m_i>n_i}m_i-n_i\\&=|\{i:m_i=n_i+1\}|+2|\{i:m_i=n_i+2\}|\\&=|\{i:m_i\ge n_i+1\}|+|\{i:m_i=n_i+2\}|\\&\le m_2^*+m_3^*,
\end{align*}
and the other inequality holds by symmetry. So we recover \cite[Theorem~3.4]{levene2023distinct} as a special case.
\end{example}

 \begin{example}\label{example:k=3,l=4}
    Consider the smallest case not covered by Corollary~\ref{cor:k-divides-l-or-l+1}, when $(k,\ell)=(3,4)$. We will show that  conditions (ii) and (iv) in Theorem~\ref{thm:main:nec-suff-q=2} are equivalent in this case. However, observe that in the case  $\m=(2^3)$ and $\n=(3,1^3)$,  condition (iv) is satisfied, but (i) is not, see also the proof of Proposition~\ref{prop:strict-inclusions}. %
    
    Observe that if $q(G\vee H)=2$ for all graphs $G,H$ which are sized by $(\m,\n)\in\bN^3\times \bN^4$, then Theorem~\ref{thm:main:nec-suff-q=2} implies 
    \begin{equation}\label{eq:k=3,l=4}
      \m\geq (2^3),\;|\n|-n_2^*-n_3^*\le |\m|\le |\n|+m_{3}^*+m_{4}^* \, \text{ and } \, \iota(\n)\leq 3.
    \end{equation}
  Let us prove that~\eqref{eq:k=3,l=4} is a sufficient condition for  $q(G\vee H)=2$ for all graphs $G,H$ which are sized by $(\m,\n)\in \bN^3\times\bN^4$. Assume without loss of generality that $\m$ and $\n$ are non-increasing. Observe first that by~\eqref{eq:k=3,l=4} we have $|\n|\geq |\m|-m_3^*-m_4^*=m_1^*+m_2^*+\sum_{j\geq 5}m_j^*\geq 6$. It follows that $\n\ge \mathbf{b}$ for some $\mathbf{b}\in \{(3,1^3),(2^2,1^2)\}$. Moreover, both possible values of $\mathbf{b}$ have $((2^3),\bfb)\in \widetilde{\LL}(3,4)$ (see the proof of Proposition~\ref{prop:strict-inclusions} for the first case; the second is easy to check).
  \begin{enumerate}[(a)]
      \item If $|\m|\geq|\n|$, then by~\eqref{eq:k=3,l=4} we have $0\leq |\m|-|\n|\leq m_3^*+m_4^*\leq 6$, so we can choose $\bdelta\in\{0,1,2\}^3$ with $\bdelta\leq (m_3^*,m_4^*)^*$ and $|\bdelta|=|\m|-|\n|$. Let $\bfp:=\m-(2^3)-\bdelta$ and 
      $\bfq:=\n-\mathbf{b}$. Observe that  $\bfp\geq (m_5^*,m_6^*,\ldots)^*$ and so $\bfp\in\bN_0^3$ and $\bfq\in \bN_0^4$   with $|\bfp|=|\bfq|$. By Lemma~\ref{lem:new1} we have $(\m-\bdelta,\n)=((2^3)+\bfp,\mathbf{b}+\bfq)\in \widetilde{\LL}(3,4)$.
      \item Similarly, if $|\m|<|\n|$, we aim to prove that $(\m,\n)=((2^3)+\bfp,\bfb+\bfq+\bepsilon)$ for one of $\mathbf{b}\in \{(3,1^3),(2^2,1^2)\}$, some $\bfp\in\bN_0^3$ and $\bfq\in \bN_0^4$   with $|\bfp|=|\bfq|$, and $\bepsilon\in\{0,1,2\}^4$. %
      Then, by Lemma~\ref{lem:new1} we will have $(\m,\n-\bepsilon)=((2^3)+\bfp,\mathbf{b}+\bfq)\in \widetilde{\LL}(3,4)$.
      Since $|\bepsilon|=|\n|-|\m|$, we need to show: 
$$\n\ge \mathbf{b}\text{ and }(\n-\mathbf{b})_1^*+(\n-\mathbf{b})_2^* \geq |\n|-|\m|$$ for one of $\mathbf{b}\in \{(3,1^3),(2^2,1^2)\}$.
Using~\eqref{eq:k=3,l=4} and $|\m|\geq 6$, the difference $|\n|-|\m|$ is bounded by two inequalities:
\begin{align*}
|\n|-|\m| &\leq n_2^*+n_3^* \\
|\n|-|\m| &\leq n_2^*+n_3^*+\sum_{j \geq 4}n_j^*-2,
\end{align*}
where the second inequality is weaker in most cases. 
Consider the following cases:
\begin{itemize}
\item If $\n \geq (3,1^3)$ and $n_5^* \geq 1$, then 
$(\n-\mathbf{b})_1^*+(\n-\mathbf{b})_2^*=n_2^*+n_3^*$
for $\mathbf{b}=(3,1^3).$
\item If $\n \geq (3,1^3)$, $n_5^* =0$, and $n_4^*\leq 1$, then $|\n|-|\m| \leq n_2^*+n_3^*+n_4^*-2$ (by the second inequality above), and 
$(\n-\mathbf{b})_1^*+(\n-\mathbf{b})_2^*=n_2^*+n_3^*+n_4^*-2$
for $\mathbf{b}=(3,1^3).$
\item If $\n \geq (3,1^3)$, $n_5^* =0$, and $n_4^*\geq 2$, then $\n \geq (2^2,1^2)$ and 
$(\n-\mathbf{b})_1^*+(\n-\mathbf{b})_2^*=n_2^*+n_3^*$
for $\mathbf{b}=(2^2,1^2)$. 
\item If  $\n  \not\geq (3,1^3)$, then $\n \geq (2^2,1^2)$ and $n_3^*=0$. In this case $|\n|-|\m| \leq n_2^*+n_3^*-2$, and 
$(\n-\mathbf{b})_1^*+(\n-\mathbf{b})_2^*=n_2^*+n_3^*-2$ for $\mathbf{b}=(2^2,1^2).$
\end{itemize}
\end{enumerate}
In both cases we have $(\m,\n)\in \LL(3,4)$ by Remark~\ref{remark:D-sets-and-AJO}(\ref{remark:D-vs-tilde(D)}) and thus $(\m,\n)$ is always join-orthogonalisable by Remark~\ref{remark:D-AJO}.
\end{example}

\begin{question}
    We end with an open question: for which pairs $(k,\ell)$ are the conditions (ii) and (iv) in Theorem~\ref{thm:main:nec-suff-q=2} equivalent? We have seen that if $k$ divides $\ell$ or $\ell+1$, or $(k,\ell)=(3,4)$, then they are; for $(k,\ell)=(5,7)$ they are not (see the proof of Proposition~\ref{prop:strict-inclusions}).
\end{question}

\section*{Acknowledgements}
Polona Oblak and Helena \v Smigoc received funding from the Slovenian Research Agency, research core funding no.~P1-0222. The authors wish to thank the anonymous referee for their valuable comments.

\bibliographystyle{amsplain}
\bibliography{references}

\end{document}